\documentclass[reqno]{amsart}

\usepackage[utf8]{inputenc}
\usepackage[T1]{fontenc}    
\usepackage{todonotes}
\usepackage{graphicx,float}
\usepackage{amsthm,amsmath,amssymb,amstext,amsfonts,tikz,caption}

\usepackage{fullpage}

\usepackage{etoolbox}
\patchcmd{\thebibliography}{\section*{\refname}}{}{}{}

\usepackage[all]{xy} 
\usepackage{ytableau}
\usepackage{amsfonts,enumerate,amsmath,amssymb,amsthm}
\usepackage{pgf,pstricks}
\usepackage{colortbl}
\usepackage[absolute]{textpos}
\usepackage{stmaryrd}
\usepackage{subcaption}
\usepackage{dsfont}
\usepackage{float}
\usepackage{pstricks,tikz,amsmath}
\usepackage{graphicx}
\usepackage{pst-grad}
\usepackage{multido}
\usepackage{tikzrput}
\usetikzlibrary{shapes,snakes}
\usetikzlibrary{positioning}
\usepackage{float}
\usepackage{genyoungtabtikz}

\newtheorem{Theorem}{Theorem}[section]

\newtheorem{Proposition}[Theorem]{Proposition}
\newtheorem{Corollary}[Theorem]{Corollary}

\theoremstyle{remark}
\newtheorem{Remark}[Theorem]{Remark}

\newtheorem{Definition}[Theorem]{Definition}

\def\cch{\mathcal H}
\def\ccsc{\mathcal {SC}}
\def\ccdd{\mathcal {DD}}
\def\ccp{\mathcal P}

\DeclareMathOperator{\BGP}{BG}

\def\bbbz{\mathbb Z}

\numberwithin{equation}{section}

\allowdisplaybreaks

\begin{document}

\title[Congruences]{Congruences for hook lengths of partitions}

\author{Fr\'{e}d\'{e}ric Jouhet}\address{Fr\'{e}d\'{e}ric Jouhet: Institut Camille Jordan, Universit\'e Claude Bernard Lyon 1,
69622 Villeurbanne Cedex, France}
\email{jouhet@math.univ-lyon1.fr}

\author{David Wahiche}\address{David Wahiche: Univ. de Tours, UMR CNRS 7013, Institut Denis Poisson, France}
 \email{wahiche@univ-tours.fr}

\keywords{Integers partitions, Littlewood decomposition, core partitions, congruences}
\subjclass[2020]{05A15, 05A17, 05A19, 05E05, 05E10, 11P81,11P83}

\maketitle
\begin{abstract}
Recently, Amdeberhan \emph{et al.} proved congruences for the number of hooks of fixed even length among the set of self-conjugate partitions of an integer $n$, therefore answering positively a conjecture raised by Ballantine \emph{et al.}. In this paper, we show how these congruences can be immediately derived and generalized from an addition theorem for self-conjugate partitions proved by the second author. We also recall how the addition theorem proved before by Han and Ji can be used to derive similar congruences for the whole set of partitions, which are originally due to Bessenrodt, and Bacher and Manivel. Finally, we extend  such congruences to the set of $z$-asymmetric partitions defined by Ayyer and Kumari, by proving an addition-multiplication theorem for these partitions. Among other things, this contains as special cases the congruences for the number of hook lengths for the self-conjugate and the so-called doubled distinct partitions.
\end{abstract}

\section{Introduction and notations}

Integer partitions are fundamental objects which, although their definition is purely combinatorial, appear in many other fields of mathematics, such as number theory, mathematical physics, and representation theory. In this note, we are interested in an important statistics regarding integer partitions, namely the hook lengths. They are involved for instance in the famous Nekrasov--Okounkov identity, discovered independently by Westbury~\cite{We} in his study of universal characters for $\mathfrak{sl}_n$, and by Nekrasov and Okounkov in their work on random partitions and Seiberg--Witten theory~\cite{NO}:
\begin{equation}\label{eq:no}
\sum_{\lambda\in\ccp}q^{|\lambda|}\prod_{h\in\mathcal{H}(\lambda)}\left(1-\frac{u}{h^2}\right)=\prod_{n\geq1}(1-q^n)^{u-1},
\end{equation}
where $u$ is any complex number and the sum is over all integer partitions $\lambda$, while $\mathcal{H}(\lambda)$ is the hook lengths multi-set of $\lambda$ (see below for precise definitions). Formula~\eqref{eq:no} was later proved and generalized in many ways (see for instance~\cite{W1} and the references cited there).

To be more precise regarding our purposes, recall that a partition $\lambda$ of a positive integer $n$ is a non-increasing sequence of positive integers $\lambda=(\lambda_1,\lambda_2,\dots,\lambda_\ell)$ such that $\lvert \lambda \rvert := \lambda_1+\lambda_2+\dots+\lambda_\ell = n$. The integers $\lambda_i$ are called the parts of $\lambda$, the number of parts $\ell$ being the length of $\lambda$, denoted by $\ell(\lambda)$. We will denote by $\ccp$ and $\ccp(n)$ the set of partitions, and its subset of partitions of $n$, respectively, and we will also use the same convention for any subset of $\ccp$.

Each partition can be represented by its Ferrers diagram, which consists in a finite collection of boxes arranged in left-justified rows, with the row lengths in non-increasing order. The Durfee square of $\lambda$ is the maximal square fitting in the Ferrers diagram. Its diagonal will be called the main diagonal of $\lambda$. Its size will be denoted $d=d(\lambda):=\max(s | \lambda_s\geq s)$. As an example, in Figure~\ref{fig:ferrers}, the Durfee square of $\lambda=(4,3,3,2)$, which is a partition of $12$ of length $4$, is coloured in red and satisfies $d(\lambda)=3$.

\begin{figure}[h]
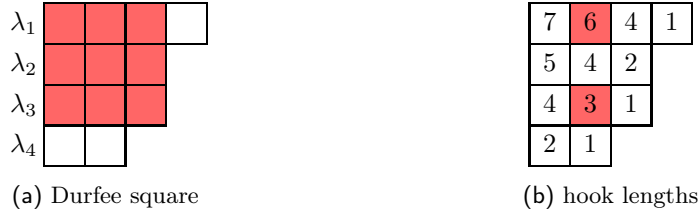

\centering
\begin{subfigure}[t]{.4\textwidth}
\centering
%
\begin{ytableau}
  \none[\lambda_1]  &  *(red!60)  &*(red!60)  & *(red!60) & \\
       \none[\lambda_2] & *(red!60) &*(red!60) &*(red!60) & \none \\
\none[\lambda_3] & *(red!60) &*(red!60)  &*(red!60) &\none \\
\none[\lambda_4] &  &  &\none & \none 
\end{ytableau}
\caption{Durfee square}
\label{fig:ferrers}
\end{subfigure}
\begin{subfigure}[t]{.4\textwidth}
\centering
\begin{ytableau}
  7  &*(red!60) 6 & 4 & 1\\
 5 &4 &2 & \none \\
 4 &*(red!60)3  &1 &\none \\
 2 & 1 &\none & \none \\
\end{ytableau}
\caption{hook lengths}
\label{fig:hooks}
\end{subfigure}
\caption{Ferrers diagram and some partition statistics}
\label{fig:fig1}
\end{figure}

Using the Durfee square, another way of describing partitions is the Frobenius notation, which is a two-rowed array representing the cells strictly to the right and below the main diagonal, namely
\begin{equation}\label{Frob}
\lambda=\left(\begin{matrix}
a_1&a_2&\dots&a_d\\
b_1&b_2&\dots&b_d
\end{matrix}\right),
\end{equation}
where $a_1>a_2>\dots >a_d\geq0$, $b_1>b_2>\dots >b_d\geq0$, and the weight is then $|\lambda|=d+\sum_k(a_k+b_k)$. By convention, the empty partition corresponds to $d=0$. In Figure \ref{fig:ferrers}, the two rows are $(3\;1\;0)$ and $(3\;2\;0)$, respectively.

For each box $v$ in the Ferrers diagram of a partition $\lambda$ (for short we will say for each box $v$ in $\lambda$), one defines the arm length (resp. leg length) as the number of boxes in the same row (resp. in the same column) as $v$ strictly to the right of (resp. strictly below) the box $v$. One defines the hook length of $v$, denoted by $h_v(\lambda)$ or $h_v$ for short, the number of boxes $u$ such that either $u=v$, or $u$ lies strictly below (resp. to the right) of $v$ in the same column (resp. row): for instance, the hooks on the main diagonal of $\lambda$ written in Frobenius notation take the form $h_{(i,i)}(\lambda)=a_i+b_i+1$, for all $1\leq i\leq d$. Moreover note that the sum of hook lengths on the main diagonal is equal to $\lvert \lambda \rvert$.
The hook length multi-set of $\lambda$, denoted by $\mathcal{H}(\lambda)$, is the multi-set of all hook lengths of $\lambda$. For any positive integer $t$, the multi-set of all hook lengths that are congruent to $0 \pmod t$ is denoted by $\mathcal{H}_t(\lambda)$. Notice that $\mathcal{H}(\lambda)=\mathcal{H}_1(\lambda)$. A partition $\omega$ is a $t$-core if $\cch_t(\omega)=\emptyset$.  In Figure~\ref{fig:hooks}, the hook lengths of all boxes for the partition $\lambda=(4,3,3,2)$ have been written in their corresponding boxes and the boxes associated with $\mathcal{H}_3(\lambda)$ shaded in red. In this example, we have $\mathcal{H}(\lambda)=\lbrace 2,1,4,3,1,5,4,2,7,6,4,1\rbrace$ and $\mathcal{H}_3(\lambda)=\lbrace 3,6\rbrace$.

 A rim hook (or border strip, or ribbon) is a connected skew shape containing no $2\times2$ square. The length of a rim hook is the number of boxes in it, and its height is one less than its number of rows. By convention, the height of an empty rim hook is zero.

Let $a$ and $q$ be complex numbers such that $\vert q \vert < 1$. Recall that the $q$-Pochhammer symbol is defined as $(a;q)_0=1$ and for any integer $n\geq 1$,
$$
(a;q)_n:= (1-a)(1-aq) \cdots (1-aq^{n-1}),\quad\mbox{and}\quad (a;q)_\infty := \prod_{j\geq 0} (1-aq^j),
$$
and more generally, we will use the compact notation  $(a_1,\dots,a_m;q)_\infty:=(a_1;q)_\infty\dots(a_m;q)_\infty$ for complex numbers $a_1,\dots,a_m$.

A classical bijection in partition theory is the Littlewood decomposition (see for instance \cite[Theorem~2.7.17]{JK}). Roughly speaking, for any positive integer $t$, it transforms $\lambda\in\ccp$ into two components, namely the $t$-core $\omega$ and the $t$-quotient $\underline{\nu}$ (see Section \ref{sec:lit} for precise definitions and properties):
$$
\lambda\in\ccp\mapsto \left(\omega,\underline{\nu}\right)\in\ccp_{(t)}\times \ccp^t.
$$
In~\cite{HJ}, Han and Ji underline some important properties of the Littlewood decomposition, which enable them to prove a multiplication-addition theorem, which specializes to the following addition theorem.

\begin{Theorem}\label{thm:addition} {\em \cite[Theorem 1.3 with $x=1$]{HJ}}
Let $t$ be a positive integer and set $\rho$ a function defined on $\mathbb{N}$. Let $g_t$ be the following formal power series:
$$
g_t(q):=\sum_{\lambda\in\mathcal{P}}q^{|\lambda|}\sum_{h\in\mathcal{H}(\lambda)}\rho(th).
$$
Then we have
\begin{equation}\label{eq:add}
\sum_{\lambda\in\mathcal{P}}q^{|\lambda|}\sum_{h\in\mathcal{H}_t(\lambda)}\rho(h)= tg_t(q^t)\frac{(q^t;q^t)_{\infty}}{(q;q)_{\infty}}.
\end{equation}
\end{Theorem}

Note that in particular setting $\rho(h)=1$ if $h=t$ and $0$ otherwise in the above theorem , and denoting by $n_t(\lambda)$ the number of hooks of length $t$ in the partition $\lambda$, we get by definition that  $g_t(q)=\sum_{\lambda\in\mathcal{P}}q^{|\lambda|}n_1(\lambda)$. Differentiating with respect to $y$ and setting $y=1$ in
\begin{equation}\label{gfn1}
\sum_{\lambda\in\mathcal{P}}q^{|\lambda|}y^{n_1(\lambda)}=\frac{((1-y)q;q)_\infty}{(q;q)_\infty},
\end{equation}
which is for instance proved elementarily  in~\cite{H}, we derive
\begin{equation}\label{ft1}
g_t(q)=\frac{q}{(1-q)(q;q)_\infty}.
\end{equation}
 Therefore~\eqref{eq:add} translates into 
$$\sum_{\lambda\in\mathcal{P}}q^{|\lambda|}n_t(\lambda)= t\,\frac{q^t}{(1-q^t)(q;q)_{\infty}}.
$$
Extracting the coefficient of $q^n$ on both sides, noticing that $a_{t}(n)=\sum_{\lambda\vdash n}n_t(\lambda)$ is the number of hooks of length $t$ among all partitions of $n$, and recalling the generating function for the set of partitions \linebreak $\sum_{\lambda\in\ccp}q^{|\lambda|}=(q;q)_\infty^{-1}$, one derives 
\begin{equation}\label{congP}
a_{t}(n)=t\sum_{j\geq1}|\ccp(n-jt)|\equiv 0\mod t.
\end{equation}

This formula for $a_t(n)$ is equivalent to a combinatorial expression by Bacher and Manivel in~\cite{BM} (which is actually a direct consequence of a result due to Bessenrodt in~\cite{Be}): for all integers $n\geq0$ and $t\geq 2$, the total number of hooks of length $t$ in all partitions of $n$ is $t$ times the total number of occurrences of the part $t$ among all partitions of $n$.

 In~\cite[Theorem 7.5]{HJ}, it is explained how choosing $\rho(h)=h^\beta$ in Theorem~\ref{thm:addition} provides generalisations of the Bacher--Bessenrodt--Manivel formula. However, one can note that our simpler choice for $\rho$ also yields Han--Ji's result, by using $\sum_{h\in\mathcal{H}_t(\lambda)}h^\beta=t^\beta\sum_{k\geq1}k^\beta n_{tk}(\lambda)$. \\

One could think that the kind of congruences appearing above in~\eqref{congP} is exceptional, but it is already known that they also occur when one restricts his attention to an important subset of $\ccp$, namely the self-conjugate partitions. Recall that the conjugate of $\lambda$, denoted $\lambda'$, is defined by its parts $\lambda_i' = \#\{j, \lambda_j \geq i\}$ for $1\leq i \leq \ell(\lambda)$. The Ferrers diagram of $\lambda'$ is thus obtained from the one of $\lambda$ by reflection with respect to the main diagonal. For instance in Figure~\ref{fig:fig1}, the conjugate of $\lambda=(4,3,3,2)$ is $\lambda'=(4,4,3,1)$.

 A partition $\lambda$ is said to be self-conjugate if it satisfies $\lambda=\lambda'$. In Frobenius notation~\eqref{Frob}, the conjugation corresponds to the inversion of the two rows, and self-conjugation is equivalent to take $a_k=b_k$ for all $k$. We denote the set of self-conjugate partitions by $\mathcal{SC}$. 

The already-mentioned Littlewood decomposition, when restricted to $\ccsc$, also has interesting properties and can be summarized as follows (see originally~\cite{Osima} and more recently for instance~\cite{GKS,W}):
$$\begin{array}{rcll}
\lambda\in\ccsc &\mapsto &\left(\omega,\underline{\tilde{\nu}}\right)\in\ccsc_{(t)}\times\ccp^{t/2}\quad&\text{if $t$ even,}\\
\lambda\in\ccsc &\mapsto &\left(\omega,\underline{\tilde{\nu}},\mu\right)\in\ccsc_{(t)}\times\ccp^{(t-1)/2}\times\ccsc\quad&\text{if $t$ odd.}\end{array}$$

Using the above Littlewood decomposition, the second author proved in~\cite{W} analogues of Theorem~\ref{thm:addition} for self-conjugate partitions (they are actually extended to addition-multiplication theorems). As one can guess, the $t$ even case was simpler to handle than the odd case: for the latter one has to restrict his attention to partitions in a subset of $\ccsc$ for which $\mu$ above is empty. It is shown in~\cite[Lemma 6.1]{W} that it corresponds to a set called $\BGP_t$ in~\cite{B} (this notation is for Brunat--Gramain, referring to~\cite{BG}) which, when $t$ is an odd prime number, is algebraically involved in representation theory of the symmetric group over a field of characteristic $t$ (see also~\cite[Lemma $3.4$]{BG}):
\begin{equation}\label{BGt}
\BGP_t:=\lbrace\lambda\in\ccsc \mid \forall i \in\lbrace 1,\dots,d\rbrace, t\nmid h_{(i,i)}(\lambda)\rbrace.
\end{equation}

Here are the addition theorems proved in~\cite{W} for self-conjugate partitions.
\begin{Theorem}\label{thm:additionSC}{\em \cite[Corollary 3.3 with $b=x=1$ and special case of Theorem 6.2]{W}}
Let $t$ be a positive integer, and let $\rho$ and $g_t$ be defined as in Theorem~\ref{thm:addition}. If $t$ is even, then we have
\begin{equation}\label{eq:addSCeven}
\sum_{\lambda\in \ccsc}q^{|\lambda|}\sum_{h\in\mathcal{H}_{t}(\lambda)}\rho(h)=tg_t(q^{2t})(q^{2t};q^{2t})_{\infty}(-q;q^2)_\infty,
\end{equation}
and if $t$ is odd, then we have
\begin{equation}\label{eq:addSCodd}
\sum_{\lambda\in \BGP_t}q^{|\lambda|}\sum_{h\in\mathcal{H}_{t}(\lambda)}\rho(h)=(t-1)g_t(q^{2t})(q^{2t};q^{2t})_{\infty}\frac{(-q;q^2)_\infty}{(-q^t;q^{2t})_\infty}.
\end{equation}
\end{Theorem}

For $t$ even, if we denote by $a_t^*(n)$ the number of hooks of length $t$ among the self-conjugate partitions of $n$, the authors of~\cite{AAOS} prove the following congruence property, which was originally conjectured in~\cite{BBCFW}.

\begin{Theorem}\label{thm:congSCeven}{\em \cite[Corollary 1.3]{AAOS}}
For all integers $n\geq0$ and $t\geq 2$ even, we have $a_{t}^*(n)\equiv 0\mod t$.
\end{Theorem}

\begin{Remark}\label{rk:unique}
There is no hope for a similar congruence for odd $t$ in general. Indeed, if $t=2t'+1$ is an odd number, note that the number of hooks of length $t$ among the self-conjugate partitions of $t$ is $1$: to see this, take $\lambda$ a self-conjugate partition such that $\lvert \lambda\rvert=t$, then the hook length of any box on the main diagonal of $\lambda$ is smaller than $\lvert \lambda\rvert$ unless $\lambda$ is the hook-shaped self-conjugate partition whose Frobenius notation is $\left(\begin{matrix}
t'\\
t'
\end{matrix}\right).$
Therefore without additional restrictions (for instance considering partitions in $\BGP_t$, or equivalently self-conjugate partitions such that their $\mu$ in the Littlewood decomposition is empty), the congruence does not hold for $n=t>1$.
\end{Remark}

In~\cite{AAOS}, the generating function of the numbers $a_t^*(n)$ is also provided, for even and odd $t$ (exhibiting no congruence in the odd case, see Remark~\ref{rk:unique}). 

In this paper, we first want to point out that for even $t$, such a generating function is an immediate consequence of a multiplication theorem also proved in~\cite{W}, while Theorem~\ref{thm:congSCeven} is derived from Theorem~\ref{thm:additionSC}~\eqref{eq:addSCeven} in the same way that~\eqref{congP} is a consequence of Theorem~\ref{thm:addition}. We will also derive analogous congruences for odd $t$, by using~\eqref{eq:addSCodd}: extend the definition of $a_t^*(n)$ to the odd $t$ case by replacing the set $\mathcal{SC}$ by $\BGP_t$.

\begin{Theorem}\label{thm:congSC}
For all integers $n\geq0$ and even $t\geq 2$, we have
$$a_{t}^*(n)=t\sum_{j\geq1}|\ccsc(n-2jt)|\quad \mbox{and}\quad a_{t}^*(n)\equiv 0\mod t,
$$
and for odd $t\geq3$, we have
$$a_{t}^*(n)=(t-1)\sum_{j\geq1}|\BGP_t(n-2jt)|\quad \mbox{and}\quad a_{t}^*(n)\equiv 0\mod (t-1).
$$
\end{Theorem}

The first part of this theorem is proved in~\cite[Theorem 1.2 (1)]{AAOS}. A similar, more complicated, expression is provided in the odd case, which, as seen in Remark~\ref{rk:unique},  exhibits no congruence.  

\begin{Remark}
As soon as $\rho(t\mathbb{N})\subseteq\mathbb{Z}$ in Theorem~\ref{thm:additionSC}, Formula~\eqref{eq:addSCeven} (resp.~\eqref{eq:addSCodd}) will provide a $0$ mod $t$ (resp. $(t-1)$) congruence when extracting coefficients of $q^n$ on both sides. This is for instance the case in our proof of Theorem~\ref{thm:congSC}.
\end{Remark}

The second goal of this paper is to extend these kinds of congruences to a larger subset of partitions. Following Ayyer--Kumari in~\cite{AK} (more precisely, we rather take the slightly different notation from~\cite{A}: partitions $\lambda$ below are conjugates of the ones in~\cite{AK}), define for any integer $z$ the set $\ccp_z$ of $z$-asymmetric partitions $\lambda$ as those whose Frobenius notation is of the form
\begin{equation}\label{Frobenius}
\lambda=\left(\begin{matrix}
a_1+z&a_2+z&\dots&a_d+z\\
a_1&a_2&\dots&a_d
\end{matrix}\right).
\end{equation}

\begin{Remark}\label{rk:pzferrers}
In terms of Ferrers diagrams, a partition $\lambda$ belongs to $\ccp_z$ if and only if $\lambda$ is made of a self-conjugate partition to which, if $z\geq0$ (resp $z<0$), a rectangle of height (resp. width) $d$ and width $z$ (resp. height $-z$) has been added to the right of (resp. below) its Durfee square of size $d$.
\end{Remark}
Thanks to this remark, one can compute the generating function of $\ccp_z$, which factorizes, either by using the $q$-exponential~\cite[(II.2)]{GR}, or by the immediate bijection between $\ccp_z$ and partitions with distinct parts of size at least $1+|z|$ and congruent to $1+|z|$ modulo $2$ (see~\cite[Corollary~6.2]{AK}, in which $z$ should be replaced by $|z|$):
$$
\sum_{\lambda\in\ccp_z}q^{|\lambda|}=\sum_{d\geq0}\frac{q^{d^2+d|z|}}{(q^2;q^2)_d}=(-q^{1+|z|};q^2)_\infty.
$$

Note that if $\lambda\in\ccp_z$, then its conjugate $\lambda'$ belongs to $\ccp_{-z}$. In~\cite{L}, Littlewood proves that the Schur function $s_\lambda$ with $tn$ variables vanishes if the $t$-core of $\lambda$ is non-empty, and otherwise it factorizes as a product of Schur functions indexed by the
partitions forming its $t$-quotient. In~\cite{AK}, Ayyer--Kumari prove
analogous factorisation theorems for the characters of the classical groups $O(2n,\mathbb{C})$, $Sp(2n,\mathbb{C})$, and $SO(2n+1,\mathbb{C})$ using Littlewood’s method: they show that the twisted characters are nonzero if and only if the $t$-core of the associated partition is in $\ccp_z$ for $z=1$, $z=-1$, and $z=0$, respectively. Note that the set of $1$-asymmetric partitions is called $\mathcal{DD}$ (for doubled distinct as illustrated in Figure \ref{fig:dd}) in~\cite{GKS}, while the set of $-1$-asymmetric partitions can be called $\mathcal{DD}'$, as its elements are conjugates of the ones in $\mathcal{DD}$, as illustrated in Figure \ref{fig:ddprime}. Of course the set of $0$-asymmetric partitions is the set $\ccsc$ of self-conjugate partitions.

\begin{figure}
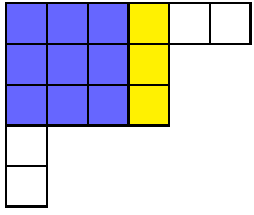
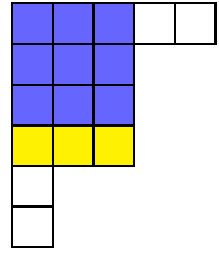

\centering
\begin{subfigure}[b]{.49\textwidth}
\centering
\begin{ytableau}
  *(blue!60)  &*(blue!60) & *(blue!60) & *(yellow)& & \\
 *(blue!60) &*(blue!60) &*(blue!60)&*(yellow) & \none \\
 *(blue!60) & *(blue!60)  &*(blue!60)&*(yellow) &\none \\
  &\none &\none &\none & \none \\
  &\none &\none &\none & \none \\
  \none & \none
\end{ytableau}
\caption{$(6,4,4,1,1)\in \ccdd$}
\label{fig:dd}
\end{subfigure}
\hfill
\begin{subfigure}[b]{.49\textwidth}
\centering
\begin{ytableau}
  *(blue!60)  &*(blue!60) & *(blue!60) & & \\
 *(blue!60) &*(blue!60) &*(blue!60)& \none \\
 *(blue!60) & *(blue!60)  &*(blue!60) &\none \\
 *(yellow)  & *(yellow) & *(yellow)\\
  &\none &\none &\none & \none \\
  &\none &\none &\none & \none \\
\end{ytableau}
\caption{$(5,3,3,3,1,1)\in \ccdd'$}
\label{fig:ddprime}
\end{subfigure}
\caption{An example of doubled distinct and its conjugate}
\end{figure}

These results are extended to the universal characters of the aforementioned groups by Albion through a new approach in~\cite{A, A1}, where a crucial tool that is proved, is the characterization of the Littlewood decomposition applied to partitions of $\ccp_z$, therefore generalizing work of Garvan--Kim--Stanton~\cite{GKS} who describe the cores for partitions in the sets $\ccp_0=\ccsc$ and $\ccp_1=\mathcal{DD}$. In Section~\ref{sec:lit}, we will describe precisely Albion's result (note that only the case $0\leq z\leq t-1$ is provided in~\cite{A}, the negative values of $z$ are then derived by conjugation using for instance~\eqref{caractmotsPzet-z} below, while Albion extends this to larger values of $z$ in~\cite{A1}, but with cumbersome modifications). 

For now, as in the above special case of self-conjugate partitions in $\ccp_0$ when $t$ is odd, we will consider a subset of $\ccp_z$ in which partitions behave in an appropriate way with respect to the Littlewood decomposition. 

\begin{Definition}\label{def:BGzt}
For integers $t\geq2$ and $0\leq z\leq t-1$, let $\BGP_{z,t}$ denote the set of partitions in $\ccp_z$ such that in their Frobenius notation~\eqref{Frobenius}, $(a_j+k)/t\notin\mathbb{N}$ and $(2a_j+z+1)/t\notin2\mathbb{N}+1$, for all $1\leq j\leq d$ and $1\leq k\leq z$. 
\end{Definition}

Note that the integers $2a_j+z+1$ are the principal hook lengths $h_{(j,j)}(\lambda)$, and that the first condition is empty for $z=0$. Therefore when $t$ is even (resp. odd), the set $\BGP_{0,t}$ is $\ccsc$ (resp. $\BGP_t$). We will prove the following addition-multiplication theorem for partitions in $\BGP_{z,t}$.

\begin{Theorem}\label{thm:addition-multz}
Let $t\geq2$ and $0\leq z\leq t-1$ be integers, and $\rho_1$, $\rho_2$ be functions defined on $\mathbb{N}$. Let $f_{t}$ and $g_{t}$ be the following formal power series:
$$
f_{t}(q):=\sum_{\lambda\in\mathcal{P}}q^{|\lambda|}\prod_{h\in\mathcal{H}(\lambda)}\rho_1(th)^2\quad\mbox{and}\quad g_{t}(q):=\sum_{\lambda\in\mathcal{P}}q^{|\lambda|}\prod_{h\in\mathcal{H}(\lambda)}\rho_1(th)^2\sum_{h\in\mathcal{H}(\lambda)}\rho_2(th).
$$
Then we have
\begin{multline}\label{eq:addz}
\sum_{\lambda\in \BGP_{z,t}}q^{|\lambda|}x^{|\mathcal{H}_t(\lambda)|}\prod_{h\in\mathcal{H}_{t}(\lambda)}\rho_1(h)\sum_{h\in\mathcal{H}_{t}(\lambda)}\rho_2(h)\\
=2\lfloor(t-z)/2\rfloor(f_t(x^2q^{2t}))^{\lfloor\frac{t-z}{2}\rfloor-1}g_t(x^2q^{2t})\prod_{i=0}^{\lfloor (t-z)/2\rfloor-1}(-q^{2i+z+1},-q^{2t-2i-z-1},q^{2t};q^{2t})_\infty.
\end{multline}
\end{Theorem}

Note that for $z=0$, $x=1$ and $\rho_1$ the constant function equal to $1$, this is Theorem~\ref{thm:additionSC} (see the generating functions in~\eqref{eq:GFz=0}). As a consequence, we also derive the following result.

\begin{Corollary}\label{coro:congz}
For all integers $t\geq2$ and $0\leq z\leq t-1$, we have the generating series
$$\sum_{\lambda\in \BGP_{z,t}}q^{|\lambda|}y^{n_{t}(\lambda)}=((1-y^2)q^{2t};q^{2t})_{\infty}^{\lfloor\frac{t-z}{2}\rfloor}\prod_{i=0}^{\lfloor (t-z)/2\rfloor-1}(-q^{2i+z+1},-q^{2t-2i-z-1};q^{2t})_\infty,$$
where $n_{t}(\lambda)$ is the number of hooks of length $t$ in $\lambda$. Moreover for all integers $n\geq0$, if $a_{z,t}(n)$ denotes the number of hooks of length $t$ among all partitions of $n$ in $\BGP_{z,t}$, then  
$$a_{z,t}(n)=2\lfloor(t-z)/2\rfloor\sum_{j\geq1}|\BGP_{z,t}(n-2tj)|\quad \mbox{and}\quad a_{z,t}(n)\equiv 0\mod (2\lfloor(t-z)/2\rfloor).
$$
\end{Corollary}

\begin{Remark}\label{rk:uniquez} 
As in Remark~\ref{rk:unique}, there is no hope for a general similar congruence for the whole set $\ccp_z$ when $t-z$ is odd. Indeed, for any $t=2m+z+1$, with $m\in\mathbb{N}$, the number of hooks of length $t$ among the $z$-asymmetric partitions of $t$ is $1$: to see this, take $\lambda\in\ccp_z$, then the hook length of any box on the main diagonal of $\lambda$ is smaller than $\lvert \lambda\rvert$ unless $\lambda$ is the hook-shaped partition whose Frobenius notation is $\left(\begin{matrix}m+z\\m\end{matrix}\right)$. Note that this partition belongs to $\ccp_z$, but not to $\BGP_{z,t}$, as it contradicts the second condition of Definition~\ref{def:BGzt}.  Similarly, one can not get a general congruence for the whole set $\ccp_z$ when $z>0$ and $t-z$ is even. However hook lengths on the principal diagonal of partitions in $\ccp_z$ cannot have the form $t=2m+z$, therefore one has to consider other weights to find a counterexample. For instance, one can prove that there is only one hook of length $t$ among the $z$-asymmetric partitions of weight $2t-z+1$. Moreover, such a hook only appears in the partition whose Frobenius notation is of the form $\left(\begin{matrix}t\\t-z\end{matrix}\right)$, which does not belong to $\BGP_{z,t}$. 
Therefore without additional restrictions (for instance considering partitions in $\BGP_{z,t}$), the congruence does not hold for $n=t>1$ (resp. $n=2t-z+1$) when $t-z$ is odd (resp. $z>0$ and $t-z$ is even).
\end{Remark}

Setting $y=1$ in the first formula of Corollary~\ref{coro:congz}, we get the generating function for our set $\BGP_{z,t}$, exhibiting that only the empty partition belongs to $\BGP_{t-1,t}$, which can also be seen directly from Definition~\ref{def:BGzt}.  Besides, the above congruence becomes Theorem~\ref{thm:congSC} when $z=0$. If moreover $t$ is even, the generating function reduces to~\cite[Theorem 1.1 (1)]{AAOS}. The special case $z=1$ corresponding to the set $\ccp_1=\mathcal{DD}$ yields elegant results similar to Theorem~\ref{thm:congSC} (and Corollary~~\ref{coro:fgSC}) for $\ccp_0=\mathcal{SC}$, therefore we gather them below.

\begin{Corollary}\label{coro:congDD}
For all integers $t\geq2$, we have 
\begin{equation}\label{BG1t}
\BGP_{1,t}=\lbrace\lambda\in\ccp_1=\mathcal{DD} \mid \forall i \in\lbrace 1,\dots,d\rbrace, t\nmid h_{(i,i)}(\lambda)\rbrace,
\end{equation}
and the generating series
$$\sum_{\lambda\in \BGP_{1,t}}q^{|\lambda|}y^{n_{t}(\lambda)}=((1-y^2)q^{2t};q^{2t})_{\infty}^{\lfloor\frac{t-1}{2}\rfloor}(-q^2;q^2)_\infty\times\left\{\displaystyle\begin{array}{ll}
(-q^{2t};q^{2t})_\infty^{-1}&\mbox{for $t$ odd}\\
\\
(-q^{t};q^{t})_\infty^{-1}&\mbox{for $t$ even,}
\end{array}\right.$$
and moreover for all integers $n\geq0$, if $a_{1,t}(n)$ denotes the number of hooks of length $t$ among all partitions of $n$ in $\BGP_{1,t}$, then 
$$a_{1,t}(n)=2\lfloor(t-1)/2\rfloor\sum_{j\geq1}|\BGP_{1,t}(n-2tj)|\quad \mbox{and}\quad a_{1,t}(n)\equiv 0\mod (2\lfloor(t-1)/2\rfloor).
$$
\end{Corollary}

Note that, although it is not immediate at first sight from Definition~\ref{def:BGzt}, the set  $\BGP_{1,t}$ is described in~\eqref{BG1t} similarly to~\eqref{BGt}. One could define the set $\BGP_{-1,t}$ whose elements are the conjugates of the partitions in $\BGP_{1,t}$, and prove similar congruence results related to hook lengths of partitions in $\ccp_{-1}=\mathcal{DD}'$. Nevertheless, it seems hopeless to get nice similar results for $|z|\geq t$ in view of~\cite[Corollary~2.4]{A1}.

Finally, we exhibit a last consequence of Theorem~\ref{thm:addition-multz}, which can be seen as a modular version for the set $\BGP_{z,t}$ of the Nekrasov--Okounkov identity~\eqref{eq:no} . The modular analogue of the latter was first shown in~\cite{H} and then derived from the multiplication theorem in~\cite{HJ}. More generally, there are many modular versions of classical combinatorial identities which are derived from the addition-multiplication theorem in~\cite{HJ}, and whose self-conjugate versions are proved in~\cite{W}. All these results could be lifted to the $\BGP_{z,t}$ case as consequences of Theorem~\ref{thm:addition-multz}, but we only highlight the following result here.

\begin{Corollary}\label{coro:NOBGzt}
For all integers $t\geq2$ and $0\leq z\leq t-1$, and any complex number $u$, we have 
$$
\sum_{\lambda\in \BGP_{z,t}}q^{|\lambda|}\prod_{h\in\mathcal{H}_{t}(\lambda)}\left(1-\frac{u}{h^2}\right)^{1/2}=(q^{2t};q^{2t})_{\infty}^{\lfloor\frac{t-z}{2}\rfloor u/t^2}\prod_{i=0}^{\lfloor (t-z)/2\rfloor-1}(-q^{2i+z+1},-q^{2t-2i-z-1};q^{2t})_\infty.$$
\end{Corollary}

This paper is organized as follows. In Section~\ref{sec:congSC} we study the case of self-conjugate partitions, proving in particular Theorem~\ref{thm:congSC}. Then in Section~\ref{sec:lit}, we recall the description of the Littlewood decomposition and we study how partitions in $\ccp_z$ and $\BGP_{z,t}$ behave under this bijection. In Section~\ref{sec:proofthm}, we prove Theorem~\ref{thm:addition-multz}, and finally in Section~\ref{sec:proofcoro} we derive Corollaries~\ref{coro:congz}--\ref{coro:NOBGzt}. 

\section{Consequences of the addition theorem and a multiplication theorem for $\mathcal{SC}$}\label{sec:congSC}

We first prove Theorem~\ref{thm:congSC} as a direct consequence of Theorem~\ref{thm:additionSC}. Then we discuss related results. Note that all the results proved in this section are consequences of our Theorem~\ref{thm:addition-multz}, but our goal is to highlight here in the case of self-conjugate partitions studied in~\cite{AAOS} how addition theorems imply immediately congruences, while multiplication theorems yield interesting generating series.

\begin{proof}[Proof of Theorem~\ref{thm:congSC}]
 Consider $\rho(h)=1$ if $h=t$ and $0$ otherwise. Then using~\eqref{ft1}, Formula~\eqref{eq:addSCeven} translates for even $t$ into 
\begin{equation}\label{addscrho1}\sum_{\lambda\in\ccsc}q^{|\lambda|}n_t(\lambda)= t\,\frac{q^{2t}}{1-q^{2t}}(-q;q^2)_\infty,
\end{equation}
while~\eqref{eq:addSCodd} translates for odd $t$ into 
$$\sum_{\lambda\in \BGP_t}q^{|\lambda|}n_t(\lambda)=(t-1)\frac{q^{2t}}{1-q^{2t}}\frac{(-q;q^2)_\infty}{(-q^t;q^{2t})_\infty},
$$
where we recall from the introduction that $n_t(\lambda)$ is the number of hooks of length $t$ in the partition $\lambda$. As for $t$ even  
$a_{t}^*(n)=\sum_{\lambda\in\ccsc(n)}n_t(\lambda)$, while for $t$ odd $a_{t}^*(n)=\sum_{\lambda\in\BGP_t(n)}n_t(\lambda)$, we get our results by extracting the coefficient of $q^n$ on both sides of these two formulas and recalling the generating functions (see for instance~\cite[(6.2)]{W} for the second one):
\begin{equation}\label{eq:GFz=0}\sum_{\lambda\in\ccsc}q^{|\lambda|}=(-q;q^2)_\infty\quad\mbox{and}\quad 
\sum_{\lambda\in\BGP_t}q^{|\lambda|}=\frac{(-q;q^2)_\infty}{(-q^t;q^{2t})_\infty}.
\end{equation}
\end{proof}

Regarding the extensions by Han--Ji of the Bacher--Bessenrodt--Manivel formula mentioned in the introduction, the second author proved similar results for $\ccsc$ in the $t$ even case in~\cite[Corollary 4.4]{W}. In the case examined here (namely $x=b=1$ and $t$ even), it is again possible to derive this particular extension from our choice of $\rho$ above:
$$
\sum_{\lambda\in \ccsc}q^{|\lambda|}\sum_{h\in\mathcal{H}_t(\lambda)}h^{\beta}=(-q;q^2)_{\infty}\sum_{k\geq 1}\left(tk\right)^{\beta+1}\frac{q^{2kt}}{1-q^{2kt}},
$$
where $\beta$ is a complex number. Indeed, write
$$
\sum_{\lambda\in \ccsc}q^{|\lambda|}\sum_{h\in\mathcal{H}_t(\lambda)}h^{\beta}=\sum_{k\geq1}(tk)^\beta\sum_{\lambda\in \ccsc}q^{|\lambda|}n_{tk}(\lambda),
$$
and use~\eqref{addscrho1} with $t$ replaced by $tk$ to conclude. Note that for $\beta=0$, the above formula yields the following congruence for $b_t^*(n)$, the number of hooks which are multiple of $t$ in all self-conjugate partitions of $n$: 
$$b_t^*(n)\equiv 0\mod t.$$
Note that this congruence is immediate by Theorem~\ref{thm:congSC}, as $b_t^*(n)=\sum_{k\geq1}a_{tk}^*(n)$. It is possible to prove similar results in the odd case, by replacing the set $\ccsc$ by $\BGP_t$.

\medskip

In~\cite{W}, the author also proves multiplication theorems for both $\ccsc$ and $\BGP_t$. We will see how the generating functions for $a_{t}^*(n)$ are immediate consequences of this result.

\begin{Theorem}{\em \cite[Corollary 3.2 with $b=x=1$ and special case of Theorem 6.2]{W}}\label{thm:multsc}
Let $t$ be a positive integer, and let $\rho$ be defined as in Theorem~\ref{thm:addition}. Let $f_t$ be the following formal power series:
$$
f_t(q):=\sum_{\lambda\in\mathcal{P}}q^{|\lambda|}\prod_{h\in\mathcal{H}(\lambda)}\rho^2(th).
$$
If $t$ is even, then we have
\begin{equation*}\label{eq:multSCeven}
\sum_{\lambda\in \ccsc}q^{|\lambda|}\prod_{h\in\mathcal{H}_{t}(\lambda)}\rho(h)=(f_{t}(q^{2t}))^{t/2}(q^{2t};q^{2t})_{\infty}^{t/2}(-q;q^2)_{\infty},
\end{equation*}
and if $t$ is odd,
\begin{equation*}\label{eq:multSCodd}
\sum_{\lambda\in \BGP_t}q^{|\lambda|}\prod_{h\in\mathcal{H}_{t}(\lambda)}\rho(h)=(f_t(q^{2t}))^{(t-1)/2}(q^{2t};q^{2t})_{\infty}^{(t-1)/2}\frac{(-q;q^2)_\infty}{(-q^t;q^{2t})_\infty}.
\end{equation*}
\end{Theorem}

\begin{Corollary}\label{coro:fgSC}
Let $t$ be a positive integer. If $t$ is even, then we have the generating function 
\begin{equation}\label{eq:gfSCeven}
\sum_{\lambda\in \ccsc}q^{|\lambda|}y^{n_{t}(\lambda)}=((1-y^2)q^{2t};q^{2t})_{\infty}^{t/2}(-q;q^2)_{\infty},
\end{equation}
and if $t$ is odd,
\begin{equation}\label{eq:gftSCodd}
\sum_{\lambda\in \BGP_t}q^{|\lambda|}y^{n_{t}(\lambda)}=((1-y^2)q^{2t};q^{2t})_{\infty}^{(t-1)/2}\frac{(-q;q^2)_\infty}{(-q^t;q^{2t})_\infty}.
\end{equation}
\end{Corollary}

\begin{proof}
Consider $\rho(h)=y$ if $h=t$ and $1$ otherwise, then by definition
\begin{equation}\label{eq:frhoy}
f_t(q)=\sum_{\lambda\in\mathcal{P}}q^{|\lambda|}y^{2n_1(\lambda)}=\frac{((1-y^2)q;q)_\infty}{(q;q)_\infty},
\end{equation}
where the second equality follows from~\eqref{gfn1}. The results are then immediate by Theorem~\ref{thm:multsc}.
\end{proof}

The generating function~\eqref{eq:gfSCeven} was established in~\cite[Theorem 1.1 (1)]{AAOS}. A more complicated expression is found in the same paper for the generating function with $t$ odd, namely:
$$\sum_{\lambda\in \ccsc}q^{|\lambda|}y^{n_{t}(\lambda)}=((1-y^2)q^{2t};q^{2t})_{\infty}^{(t-1)/2}\frac{(-q;q^2)_\infty}{(-q^t;q^{2t})_\infty}\sum_{\lambda\in\ccsc}q^{|\lambda|}y^{n_1(\lambda)}.$$
This formula is also proved by using the Littlewood decomposition.  Nevertheless we find more appropriate to exhibit~\eqref{eq:gftSCodd} instead, which shows that the correct point of view towards finding factorized generating functions  (resp. congruences) is the multiplication (resp. addition) theorem provided by the Littlewood decomposition restricted to $\BGP_t$: the complicated generating function $\sum_{\lambda\in\ccsc}q^{|\lambda|}y^{n_1(\lambda)}$, expressed in~\cite[Theorem~3.1]{AAOS} is not any more needed, and one can get congruences as in Theorem~\ref{thm:congSC} whereas it is not the case for $\ccsc$ with $t$ odd (see Remark~\ref{rk:unique}).

\section{Combinatorial properties of the Littlewood decomposition on $z$-asymmetric partitions} \label{sec:lit}

In this section, we use the formalism of Han and Ji in \cite{HJ}. Recall that a partition $\mu$ is a $t$-core if it has no hook that is a multiple of $t$. For any $A\subset\ccp$, we denote by $A_{(t)}$ the subset of elements of $A$ that are $t$-cores. For example, the only $2$-cores are the ``staircase" partitions $(k,k-1,\dots,1)$, for any positive integer $k$, which are also the only $\ccsc$ $2$-cores.

Let $\partial \lambda$ be the border of the Ferrers diagram of $\lambda$. Each step on $\partial\lambda$ is either horizontal or vertical. Encode the walk along the border from the South-West to the North-East as depicted in Figure \ref{fig:word}: take ``$0$" for a vertical step and ``$1$" for a horizontal step. This yields a $0/1$ sequence denoted \textbf{$s(\lambda)$}. The resulting word $s(\lambda)$ over the $\lbrace 0,1\rbrace$ alphabet:
\begin{itemize}
\item  contains infinitely many ``$0$"'s (respectively ``$1$"'s) at the beginning (respectively the end),
\item is indexed by $\bbbz$,
\item  and is written $(c_i)_{i\in\mathbb{Z}}$.  
\end{itemize}
This writing as a sequence is not unique since for any $k$, sequences $(c_{k+i})_{i\in\mathbb{Z}}$ encode the same partition. Hence it is necessary for that encoding to be bijective to set the index $0$ uniquely. To tackle that issue, we set the index $0$ when the number of ``$0$"'s on and to the right of that index is equal to the number of ``$1$"'s to the left. In other words, the number of horizontal steps along $\partial\lambda$ corresponding to a ``$1$" of negative index in $(c_i)_{i\in\bbbz}$ must be equal to the number of vertical steps corresponding to ``$0$"'s of non-negative index in $(c_i)_{i\in\bbbz}$ along $\partial\lambda$. The delimitation between the letter of index $-1$ and the one of index $0$ is called the median of the word, marked by a $\mid$ symbol. The size of the Durfee square is then equal to the number of ``$1$"'s of negative index (equivalently the number of ``$0$"'s of positive index). Hence a partition is bijectively associated by the application $s$ to the word:
\begin{align*}
s(\lambda)=(c_i)_{i\in\mathbb{Z}}=\left(\ldots c_{-2}c_{-1}|c_0c_1c_2\ldots\right), \intertext{where $c_i\in\lbrace 0,1\rbrace$ for any $i\in\bbbz$, and such that}
\#\{i\leq-1,c_i=1\}= \#\{i\geq0,c_i=0\}.
\end{align*}

Moreover, this application maps bijectively a box $u$ of hook length $h_u$ of the Ferrers diagram of $\lambda$ to a pair of indices $(i_u,j_u)\in\bbbz^2$ of the word $s(\lambda)$ such that
\begin{equation*}\label{conditionshookmot}
i_u<j_u,\quad c_{i_u}=1,\quad  c_{j_u}=0,\quad j_u-i_u=h_u.
\end{equation*} 
Recall for instance from~\cite[Definition~6.1]{LW} that for all $i\in\mathbb{Z}$, we have
\begin{equation*}
c_i=\left\{\begin{matrix}
0&\mbox{if}&i\in\{\lambda_j-j,\,j\in\mathbb{N}\},\\
1&\mbox{if}&i\in\{j-\lambda'_j-1,\,j\in\mathbb{N}\}.
\end{matrix}\right.
\end{equation*} 

Thanks to this, it is possible to connect $s(\lambda)$ with the Frobenius notation~\eqref{Frob} of $\lambda$, as also done using abacus in \cite{BN}:
\begin{equation}\label{conditionshookmotFrob}
\lbrace i \in\mathbb{N}, c_i=0\rbrace=\lbrace a_j, j\in\lbrace 1,\dots,d\rbrace \rbrace\text{ and }
\lbrace -i \in\mathbb{N}^*, c_{i}=1\rbrace=\lbrace -b_j-1, j\in\lbrace 1,\dots,d\rbrace \rbrace.
\end{equation}

\begin{figure}[h!]
\centering
\begin{tikzpicture}
    [
        dot/.style={circle,draw=black, fill,inner sep=1pt},
    ]

\foreach \x in {0,...,2}{
    \node[dot] at (\x,-4){ };
}

\foreach \x in {.2,1.2}
    \draw[->,thick,orange] (\x,-4) -- (\x+.6,-4);
\foreach \x in {3.2,4.2}
    \draw[->,thick,orange] (\x,-2) -- (\x+.6,-2);
\foreach \y in {1.8,.8}
    \draw[->,thick,orange] (5,-\y) -- (5,-\y+.6);
\draw[->,thick,orange] (2,-3.8) -- (2,-3.8+.6);
\draw[->,thick,orange] (2,-2.8) -- (2,-2.8+.6);
\draw[->,thick,orange] (2.2,-2) -- (2.2+.6,-2);

\node[dot] at (2,-3){};
\foreach \x in {3,...,5}
    \node[dot] at (\x,-2){};
\foreach \x in {1,...,4}
    \draw (\x,-.1) -- node[above,xshift=-0.4cm,yshift=1mm] {$\lambda_\x'$} (\x,+.1);

\node[above,xshift=0.5cm,yshift=1mm] at (4,0) {$\lambda_5'$};
\node[above,xshift=0.5cm,yshift=1mm] at (5,0) {NE};
\node[above,xshift=-4mm,yshift=1mm] at (0,0) {NW};

\foreach \y in {1,...,3}
    \draw (.1,-\y) -- node[above,xshift=-4mm,yshift=0.2cm] {$\lambda_\y$} (-.1,-\y);
\node[above,xshift=-4mm,yshift=2mm] at (0,-4) {$\lambda_4$};
\node[above,xshift=-4mm] at (0,-5) {SW};
\node at (0,-4.5) {0};  
\node at (0,-5.5) {0};
\node at (2,-3.5) {0};
\node at (2,-2.5) {0};
\node at (5,-1.5) {0};
\node at (5,-0.5) {0};

\node at (0.5,-4) {1};
\node at (1.5,-4) {1};   
\node at (4.5,-2) {1};
\node at (3.5,-2) {1};   
\node at (5.5,-0) {1};
\node at (2.5,-2) {1};

\node[dot] at (5,-1){};  
\node[dot] at (5,0){};
\draw[->,thick,-latex] (0,-6) -- (0,-5);
\draw[thick] (0,-6) -- (0,1);
\draw[->,thick,-latex] (-1,0) -- (6,0);
\node[circle,draw=blue,fill=blue,inner sep=0pt,minimum size=5pt] at (2,-2){};

\end{tikzpicture}
\caption{For $\lambda=(5,5,2,2)$, $s(\lambda)=\dots01100\mid 111001\dots$.}
\label{fig:word}
\end{figure}
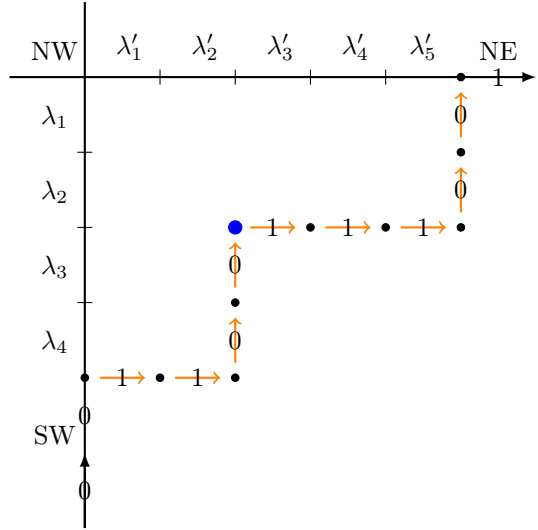
Now we recall the following classical map, often called the Littlewood decomposition (see for instance \cite{GKS,HJ}). Let $t \geq 2$ be an integer and consider:\\
$$\begin{array}{l|rcl}
\Phi_t: & \mathcal{P} & \to & \mathcal{P}_{(t)} \times \mathcal{P}^t \\
& \lambda & \mapsto & (\omega,\nu^{(0)},\ldots,\nu^{(t-1)}),
\end{array}$$
where if we set $s(\lambda)=\left(c_i\right)_{i\in\bbbz}$, then for all $k\in\lbrace 0,\dots,t-1\rbrace$, one has $\nu^{(k)}:=s^{-1}\left(\left(c_{ti+k}\right)_{i\in\bbbz}\right)$. The tuple $\underline{\nu}=\left(\nu^{(0)},\ldots,\nu^{(t-1)}\right)$ is called the $t$-quotient of $\lambda$ and is denoted by \textit{$quot_t(\lambda)$}, while $\omega$ is the $t$-core of $\lambda$ denoted by \textit{$core_t(\lambda)$}. Obtaining the $t$-quotient is straightforward from $s(\lambda)=\left(c_i\right)_{i\in\bbbz}$: we just look at subwords with indices congruent to the same values modulo $t$. The sequence $10$ within these subwords are replaced iteratively by $01$ until the subwords are all the infinite sequence of ``$0$"'s before the infinite sequence of ``$1$"'s (in fact it consists in removing all rim hooks in $\lambda$ of length congruent to $0\pmod t$). Then $\omega$ is the partition corresponding to the word which has the subwords$\pmod t$ obtained after the removal of the $10$ sequences.

For example, if we take $\lambda = (5,5,2,2)$ as in Figure~\ref{fig:word} and $t=3$, then $s(\lambda)=\ldots\color{green}{0}\color{red}{0}\color{blue}{0}\color{green}{1}\color{red}{1}\color{blue}{0}\color{green}{0} \color{black}| \color{red}{1}\color{blue}{1}\color{green}{1}\color{red}{0}\color{blue}{0}\color{green}{1}\color{red}{1}\color{blue}{1}\ldots $
\begin{align*}
\begin{array}{rc|rcl}
s\left(\nu^{(0)}\right)=\ldots \color{red} 001 \color{black}| \color{red}101\color{black}\ldots& &s\left(w_{0}\right)=\ldots \color{red} 000 \color{black}| \color{red}111\color{black}\ldots,\\
 s\left(\nu^{(1)}\right)=\ldots \color{blue} 000 \color{black}| \color{blue}101\color{black}\ldots& \longmapsto& s\left(w_{1}\right)=\ldots \color{blue} 000 \color{black}| \color{blue}011\color{black}\ldots , \\
 s\left(\nu^{(2)}\right)=\ldots \color{green} 010 \color{black}| \color{green}111\color{black}\ldots & & s\left(w_{2}\right)=\ldots \color{green} 001 \color{black}| \color{green}111\color{black}\ldots .
\end{array}\\
\end{align*}
Thus
$$s(\omega)=\ldots \color{red}{0} \color{blue}{0} \color{green}{0} \color{red}{0} \color{blue}{0}\color{green}{1} \color{black}| \color{red}{1} \color{blue}{0} \color{green}{1} \color{red}{1} \color{blue}{1} \color{green}{1}\color{black}\ldots\quad\mbox{so}\quad core_3(\lambda)= \omega=(2),$$

and
$$
quot_3(\lambda)=\left(\nu^{(0)},\nu^{(1)},\nu^{(2)}\right)=\left((2),(1),(1)\right).
$$

Note that $\lambda=(5,5,2,2)$ in the above example is in $\mathcal{DD}=\ccp_1$, and, using Definition~\ref{def:BGzt} or~\eqref{BG1t}, neither belongs to $\BGP_{1,3}$ nor $\BGP_{1,4}$  but we have $\lambda\in\BGP_{1,5}$.

The following properties of the Littlewood decomposition are given in \cite{HJ}.

\begin{Proposition}\cite[Theorem 2.1]{HJ}\label{prop:hanlitt} Let $t$ be a positive integer. The Littlewood decomposition $\Phi_t$ maps bijectively a partition $\lambda$ to $\left(\omega,\nu^{(0)},\dots,\nu^{(t-1)}\right)$ such that:
\begin{align*}
&(P1)\quad \omega \text{ is a $t$-core and }\nu^{(0)},\dots,\nu^{(t-1)} \text{are partitions},\\
&(P2) \quad |\lambda|=|\omega|+t\sum_{i=0}^{t-1} |\nu^{(i)}|,\\
&(P3)\quad \mathcal{H}_t(\lambda)=t\mathcal{H}(\underline{\nu}),
\end{align*}
where, for a multiset $S$, $tS:=\{ts,s\in S\}$ and $\mathcal{H}(\underline{\nu}):=\bigcup\limits_{i=0}^{t-1}\cch(\nu^{(i)})$.
\end{Proposition}

Note that for any partition $\lambda$, if $s(\lambda)=(c_i)_{i\in\mathbb{Z}}$ and $s(\lambda')=(c'_i)_{i\in\mathbb{Z}}$, then $c'_i=1-c_{-i-1}$ for all $i\in\mathbb{Z}$ (see for instance~\cite[Remark 6.3]{LW}). Therefore using Remark~\ref{rk:pzferrers}, we have the characterizations
\begin{equation}\label{caractmotsPzet-z}\lambda\in\ccp_z\Leftrightarrow\left\{\begin{matrix}
c_0=c_1=\dots=c_{z-1}=1\\
c_i=1-c_{z-i-1}\quad\forall i\geq z
\end{matrix}\right.\quad\mbox{and}\quad\lambda\in\ccp_{-z}\Leftrightarrow\left\{\begin{matrix}
c_{-1}=c_{-2}=\dots=c_{-z}=0\\
c_i=1-c_{-z-i-1}\quad\forall i\leq -z-1.
\end{matrix}\right.
\end{equation}

Now following~\cite{A}, we describe the image $\Phi_t(\ccp_z)$, where $0\leq z\leq t-1$ are fixed integers. As recalled in~\cite[(2.2)]{A}, there is in~\cite{GKS} a bijective correspondence $\kappa_t$ between $t$-cores $\omega$ of partitions and vectors of integers $\mathbf{n}\in\mathbb{Z}^t$ such that $n_0+n_1+\dots+n_{t-1}=0$. Using the above description of $\omega$ in terms of sequences of $0$'s and $1$'s, namely $s(\omega)=(c_i)_{i\in\mathbb{Z}}$ , it is possible to define $\kappa_t(\omega)=\mathbf{n}$ by
$$n_i:=\min\{k\in\mathbb{Z}\,|\,c_{i+kt}=1\},\quad 0\leq i\leq t-1.$$
In other words, $n_i$ is the index of the first $1$ appearing in the sub-word of $s(\omega)$ corresponding to indices congruent to $i$ modulo $t$. Following~\cite{A}, set
$$\mathcal{C}_{z;t}:=\{\mathbf{n}\in\mathbb{Z}^t\,|\,n_r+n_{z-r-1}=0\;\mbox{for}\;0\leq r\leq z-1,\;\mbox{and}\;n_r+n_{t+z-r-1}=0\;\mbox{for}\;z\leq r\leq t-1\},$$
and for a partition $\lambda$, define $d_c(\lambda)$ as the size of the Durfee square of the partition obtained from $\lambda$ by removing its first $c$ parts.

\begin{Theorem}\label{thm:albion}{\em \cite[Theorem 2]{A}}
Let $t\geq2$ and $0\leq z\leq t-1$ be integers, and $\lambda\in\ccp_z$. Setting $\Phi_t(\lambda)=(\omega,\nu^{(0)},\dots,\nu^{(t-1)})$, we have $\kappa_t(\omega)\in\mathcal{C}_{z;t}$ and for $0\leq r\leq z-1$, if $n_r\geq0$, then there exists a partition $\mu^{(r)}$ such that
\begin{equation}\label{eq:albion1}
\nu^{(r)}=\mu^{(r)}+(1^{n_r+d_{n_r}(\mu^{(r)})})\quad\mbox{and}\quad\nu^{(z-r-1)}=(\mu^{(r)})'+(1^{d_{n_r}(\mu^{(r)})}),
\end{equation}
while for $z\leq r\leq t-1$, we have
\begin{equation}\label{eq:albion2}
\nu^{(r)}=(\nu^{(t+z-r-1)})'.
\end{equation}
\end{Theorem}

As for the special case $z=0$ corresponding to self-conjugate partitions (when $t$ is odd), we need to restrict the above conditions~\eqref{eq:albion1} and~\eqref{eq:albion2} to be able to derive addition-multiplication theorems, and congruences. The following result makes the connection with our sets $\BGP_{z,t}$ defined in the introduction.
\begin{Proposition}\label{prop:bgzt}
Let $t\geq2$ and $0\leq z\leq t-1$ be integers. A partition $\lambda$ belongs to the set $\BGP_{z,t}$ from Definition~\ref{def:BGzt} if and only if setting $\Phi_t(\lambda)=(\omega,\nu^{(0)},\dots,\nu^{(t-1)})$, we have $\kappa_t(\omega)\in\mathcal{C}_{z;t}$, $\nu^{(r)}=\emptyset$ for $0\leq r\leq z-1$, and if $t+z-1$ is even, then $\nu^{((t+z-1)/2)}=\emptyset$.
\end{Proposition}

\begin{proof}
A partition is empty if and only if its corresponding sequence $(c_i)_{i\in\mathbb{Z}}$ is made of infinitely many $0$'s followed by infinitely many $1$'s. Therefore, using the Frobenius notation~\eqref{Frobenius} for any partition $\lambda\in\ccp_z$, the first case of~\eqref{conditionshookmotFrob} ensures that $\nu^{(r)}=\emptyset$ for all $0\leq r\leq z-1$ in $\Phi_t(\lambda)$ if and only if $ti+r\notin\{a_1+z,\dots,a_d+z\}$ for all $0\leq r\leq z-1$ and $i\geq0$. Setting $k=z-r$, this is equivalent to $(a_j+k)/t\notin\mathbb{N}$ for all $1\leq j\leq d$ and $1\leq k\leq z$.

Similarly, when $t+z-1$ is even, $\nu^{((t+z-1)/2)}=\emptyset$ if and only if $ti+(t+z-1)/2\notin\{a_1+z,\dots,a_d+z\}$ for all $i\geq0$, which is equivalent to $t\nmid(a_j+z-(t+z-1)/2)$ for all $1\leq j\leq d$. Meanwhile, for a $j_0\in\{1,\dots,d\}$, we have 
\begin{eqnarray*}
t\mid\left(a_{j_0}+z-\frac{t+z-1}{2}\right)&\Leftrightarrow &\exists\alpha_0\in\mathbb{N},\,2a_{j_0}+2z-t-z+1=2\alpha_0t\\
&\Leftrightarrow &\frac{2a_{j_0}+z+1}{t}\in 2\mathbb{N}+1.
\end{eqnarray*}

When $t+z-1$ is odd, the integers $z$ and $t$ have the same parity. If we had $2a_{j_0}+z+1=(2\alpha_0+1)t$ for integers $\alpha_0\geq0$ and $1\leq j_0\leq d$, then $z$ and $t$ would have different parity, which is a contradiction. Therefore the second condition in Definition~\ref{def:BGzt} is always satisfied when $t+z-1$ is odd, and we get the conclusion. 

\end{proof}

Note that thanks to Proposition~\ref{prop:bgzt}, the $t$-quotients of partitions in $\Phi_t(\BGP_{z,t})$ are such that in~\eqref{eq:albion1} only empty partitions appear, and there is at least one more empty partition in~\eqref{eq:albion2} if $t+z-1$ is even. 

\begin{Remark}\label{rk:corebgzt}
As proved by Albion in~\cite[Corollary 3]{A}, a $t$-core is in $\ccp_z$ if and only if $n_r=0$ for all $0\leq r\leq z-1$. As a consequence, the $t$-core of any partition in $\BGP_{z,t}$ is itself in $\ccp_z$, which means that the set of $\BGP_{z,t}$ $t$-cores is the one of $\ccp_z$ $t$-cores .
\end{Remark}

Hence the analogue of Proposition~\ref{prop:hanlitt} when applied to partitions in $\BGP_{z,t}$ is as follows.

\begin{Proposition}\label{prop:bgztLittlewood}
Let $t\geq2$ and $0\leq z\leq t-1$ be integers. The Littlewood decomposition $\Phi_t$ maps a partition $\lambda\in\BGP_{z,t}$ to $\left(\omega,\nu^{(0)},\dots,\nu^{(t-1)}\right)=(\omega,\underline{\nu})$ such that:
\begin{align*}
&(BG1)\quad \text{the first component } \omega \text{ is a $\ccp_z$ $t$-core and }\nu^{(0)},\dots,\nu^{(t-1)} \text{are partitions},\\
&(BG2) \quad \forall r \in \left\lbrace z,\dots,t-1\right\rbrace, \nu^{(r)}=\left(\nu^{(t+z-r-1)}\right)',\\
&(BG'2)\quad \nu^{(0)}=\dots=\nu^{(z-1)}=\emptyset\quad \text{and if $t+z-1$ is even, } \nu^{\left((t+z-1)/2\right)}=\emptyset, \\
&(BG3) \quad |\lambda|=|\omega|+2t\sum_{r=z}^{\lfloor (t+z-2)/2\rfloor} \lvert\nu^{(r)}\rvert,\\
&(BG4)\quad \mathcal{H}_t(\lambda)=t\mathcal{H}(\underline{\nu}).
\end{align*}
\end{Proposition}

Therefore $\lambda\in\BGP_{z,t}$ is uniquely defined if its $t$-core is known as well as the $\lfloor (t-z)/2\rfloor$ elements starting from $\nu^{(z)}$ of its quotient, which are partitions without any constraints. 

\section{Proof of the addition-multiplication theorem for $z$-asymmetric partitions}\label{sec:proofthm}

In this section, we prove Theorem~\ref{thm:addition-multz} stated in the introduction. First we will compute the term
\begin{equation}
\displaystyle\sum_{\substack{\lambda\in \BGP_{z,t}\\core_t(\lambda)=\omega}}q^{|\lambda|}x^{|\mathcal{H}_{t}(\lambda)|}\prod_{h\in\mathcal{H}_{t}(\lambda)}\rho_1(h)\sum_{h\in\mathcal{H}_{t}(\lambda)}\rho_2(h),\label{term}
\end{equation}
where $\omega$ is a fixed $\BGP_{z,t}$ $t$-core, equivalently (by Remark~\ref{rk:corebgzt}) a $\ccp_z$ $t$-core. Using properties $(P2)$ and $(P3)$ from Proposition~\ref{prop:hanlitt}, this is equal to
\begin{equation}
q^{|\omega|}\sum_{\underline{\nu}\in \ccp^t} q^{t\displaystyle\lvert\underline{\nu}\rvert}x^{\displaystyle\lvert\underline{\nu}\rvert}\prod_{h\in\mathcal{H}(\underline{\nu})}\rho_1(th)\sum_{h\in\mathcal{H}(\underline{\nu})}\rho_2(th),\label{sumprod}
\end{equation}
where $\lvert\underline{\nu}\rvert:=\sum_{i=0}^{t-1}\lvert\nu^{(i)}\rvert$. Now by properties $(BG2)$ and $(BG'2)$ from Proposition~\ref{prop:bgztLittlewood}, for $0\leq r\leq z-1$ we have $\nu^{(r)}=\emptyset=\nu^{((t+z-1)/2)}$ and for $z\leq r\leq \lfloor (t+z-2)/2\rfloor$ we have $|\nu^{(r)}|=|\nu^{(t+z-r-1)}|$ and $\mathcal{H}(\nu^{(r)})=\mathcal{H}(\nu^{(t+z-r-1)})$ because sizes and hook lengths multisets of partitions are invariant by conjugation.

Therefore the product in~\eqref{sumprod} can be rewritten as follows
$$
q^{t\lvert\underline{\nu}\rvert}x^{\lvert\underline{\nu}\rvert}\prod_{h\in\mathcal{H}(\underline{\nu})}\rho_1(th)= \prod_{r=z}^{\lfloor (t+z-2)/2\rfloor}q^{2t\lvert\nu^{(r)}\rvert}x^{2\lvert\nu^{(r)}\rvert}\prod_{h\in\mathcal{H}(\nu^{(r)})}\rho_1^2(th).
$$

Moreover by application of Proposition~\ref{prop:bgztLittlewood}~$(BG2)$ and~$(BG'2)$, the sum part $\sum_{h\in\mathcal{H}(\underline{\nu})}\rho_2(th)$ in~\eqref{sumprod} is 
$$
\sum_{r=z}^{\lfloor (t+z-2)/2\rfloor}\left(\sum_{h\in\mathcal{H}(\nu^{(r)})}\rho_2(th)+\sum_{h\in\mathcal{H}(\nu^{(t-r-1)})}\rho_2(th)\right)=2\sum_{r=z}^{\lfloor (t+z-2)/2\rfloor}\sum_{h\in\mathcal{H}(\nu^{(r)})}\rho_2(th).
$$

Therefore \eqref{sumprod}, and thus \eqref{term}, becomes
\begin{multline*}
2q^{\vert \omega \vert}\sum_{r=z}^{\lfloor (t+z-2)/2\rfloor}\left(\sum_{\nu^{(r)}\in\ccp}\displaystyle q^{2t\lvert\nu^{(r)}\rvert}x^{2\lvert\nu^{(r)}\rvert}\prod_{h\in\mathcal{H}(\nu^{(r)})}\rho_1^2(th)\sum_{h\in\mathcal{H}(\nu^{(r)})}\rho_2(th)\right)\\
\times\left(\displaystyle\sum_{\nu\in\ccp} q^{2t\lvert\nu\rvert}x^{2\lvert\nu\rvert}\prod_{h\in\mathcal{H}(\nu)}\rho_1^2(th)\right)^{\lfloor (t-z)/2\rfloor-1}.
\end{multline*}

Hence we get
$$
\sum_{\substack{\lambda\in \BGP_{z,t}\\core_t(\lambda)=\omega}}q^{|\lambda|}x^{|\mathcal{H}_{t}(\lambda)|}\prod_{h\in\mathcal{H}_{t}(\lambda)}\rho_1(h)\sum_{h\in\mathcal{H}_{t}(\lambda)}\rho_2(h)=2\left(\lfloor (t-z)/2\rfloor\right)q^{|\omega|}\left(f_t\left(x^2q^{2t}\right)\right)^{\lfloor (t-z)/2\rfloor-1}g_t(x^2q^{2t}).
$$

To finish the proof, it remains to sum both sides over all $\ccp_z$ $t$-core partitions $\omega$ and compute the generating function $\sum_\omega q^{|\omega|}$, where the sum is over all $\ccp_z$ $t$-cores. This was done in~\cite[Corollary~6.4]{AK} (we just consider the empty product as $1$ to include the case $z=t-1$ which was omitted in the formula from~\cite{AK}, as it is immediate by the first equivalence in~\eqref{caractmotsPzet-z} and the definition of $\mathcal{C}_{t-1,t}$ that the only $\ccp_{t-1}$ $t$-core is the empty partition):
\begin{equation}\label{AKPzcores}
\sum_{\substack{\omega\in \ccp_z\\\omega\,t-core}} q^{|\omega|}=\prod_{i=0}^{\lfloor (t-z)/2\rfloor-1}(-q^{2i+z+1},-q^{2t-2i-z-1},q^{2t};q^{2t})_\infty,
\end{equation}
and this gives the desired result.

\section{Some consequences}\label{sec:proofcoro}

Here we prove Corollaries~\ref{coro:congz}--\ref{coro:NOBGzt} from the introduction.

\begin{proof}[Proof of Corollary~\ref{coro:congz}]

First taking $\rho_2(h)=1$ in Theorem~\ref{thm:addition-multz} yields $g_t(q)=\sum_{\lambda\in\ccp}|\lambda|q^{|\lambda|}\prod_{h\in\mathcal{H}(\lambda)}\rho_1(th)^2$, therefore 
\begin{equation}\label{eq:rho2=1}
g_t(x^2q^{2t})=\frac{x}{2}\frac{d}{dx}f_t(x^2q^{2t}).
\end{equation}
Considering moreover $\rho_1(h)=y$ if $h=t$ and $1$ otherwise, then by~\eqref{eq:frhoy}, we get
$$f_t(q)=\sum_{\lambda\in\mathcal{P}}q^{|\lambda|}y^{2n_1(\lambda)}=\frac{((1-y^2)q;q)_\infty}{(q;q)_\infty}.$$
Thus using~\eqref{eq:rho2=1}, Formula~\eqref{eq:addz} becomes
\begin{multline*}
\sum_{\lambda\in\BGP_{z,t}}q^{|\lambda|}x^{|\mathcal{H}_t(\lambda)|}|\mathcal{H}_t(\lambda)|\,y^{n_t(\lambda)}=x\lfloor(t-z)/2\rfloor(f_t(x^2q^{2t}))^{\lfloor\frac{t-z}{2}\rfloor-1}\frac{d}{dx}f_t(x^2q^{2t})\\
\times(q^{2t};q^{2t})^{\lfloor\frac{t-z}{2}\rfloor}_{\infty}\prod_{i=0}^{\lfloor (t-z)/2\rfloor-1}(-q^{2i+z+1},-q^{2t-2i-z-1};q^{2t})_\infty.
\end{multline*}
The desired generating series follows by dividing both sides by $x$, integrating with respect to $x$, and setting $x=1$.

For the second part of the corollary, we choose this time $\rho_1$ the constant function equal to $1$ and $\rho_2(h)=1$ if $h=t$, $0$ otherwise. Therefore we get $f_{t}(q)=1/(q;q)_\infty$ and by~\eqref{ft1},
$$
g_t(q)=\frac{q}{(1-q)(q;q)_\infty}.
$$
Thus, setting $x=1$, Formula~\eqref{eq:addz} becomes this time
$$
\sum_{\lambda\in\BGP_{z,t}}q^{|\lambda|}n_t(\lambda)=2\lfloor(t-z)/2\rfloor\frac{q^{2t}}{1-q^{2t}}\prod_{i=0}^{\lfloor (t-z)/2\rfloor-1}(-q^{2i+z+1},-q^{2t-2i-z-1};q^{2t})_\infty.
$$
As $a_{z,t}(n)=\sum_{\lambda\in\BGP_{z,t}(n)}n_t(\lambda)$ and the last product is the generating function of $\BGP_{z,t}$, we get our results by extracting the coefficient of $q^n$ on both sides of the above formula.
\end{proof}

\begin{proof}[Proof of Corollary~\ref{coro:congDD}]

The Littlewood decomposition, when restricted to $\mathcal{DD}=\ccp_{1}$, has well-known properties (see for instance~\cite{GKS}), which can be recovered from Theorem~\ref{thm:albion}. For short, one gets:
$$\begin{array}{rcll}
\lambda\in\mathcal{DD} &\mapsto &\left(\omega,\underline{\tilde{\nu}},\eta\right)\in\mathcal{DD}_{(t)}\times\ccp^{(t-1)/2}\times\mathcal{DD}\quad&\text{if $t$ odd,}\\
\lambda\in\mathcal{DD} &\mapsto &\left(\omega,\underline{\tilde{\nu}},\eta,\mu\right)\in\mathcal{DD}_{(t)}\times\ccp^{t/2-1}\times\mathcal{DD}\times\ccsc\quad&\text{if $t$ even.}\end{array}$$

By Proposition~\ref{prop:bgzt} the set $\BGP_{1,t}$ is simply the set of partitions in $\mathcal{DD}$ for which the above $\eta$ and $\mu$ are empty. 
By Definition~\ref{def:BGzt} with $z=1$, a partition $\lambda$ with Frobenius notation~\eqref{Frob} belongs to $\BGP_{1,t}$ if and only if for all $j\in\{1,\dots,d\}$, $(a_j+1)/t\notin\mathbb{N}$ and $(2a_j+2)/t\notin2\mathbb{N}+1$. As the principal hooks are $h_{(i,i)}(\lambda)=2a_i+2$, this second condition is equivalent to  $h_{(i,i)}(\lambda)/t\notin 2\mathbb{N}+1$ for all $i\in\lbrace 1,\dots,d\rbrace$. Now for $j_0\in\{1,\dots,d\}$, we have 
$$
\frac{a_{j_0}+1}{t}\in \mathbb{N}\Leftrightarrow \exists\alpha_0\in\mathbb{N},\,2a_{j_0}+2=2\alpha_0t\Leftrightarrow \frac{h_{(j_0,j_0)}(\lambda)}{t}\in 2\mathbb{N}.
$$
Therefore $\lambda\in\BGP_{1,t}$ if and only if  $h_{(i,i)}(\lambda)/t$ belongs neither to $ 2\mathbb{N}+1$ nor to $2\mathbb{N}$ for all $i\in\lbrace 1,\dots,d\rbrace$, which proves~\eqref{BG1t}.

The generating series and the congruences are then immediate consequences of Corollary~\ref{coro:congz}, together with classical $q$-series manipulations: in the $t$ odd case, one can for instance write
\begin{align*}
\prod_{i=0}^{(t-1)/2-1}(-q^{2i+z+1},-q^{2t-2i-z-1},q^{2t};q^{2t})_\infty&=(-q^2,-q^4,\dots,-q^{t-1},-q^{t+1},\dots,-q^{2t-2};q^{2t})_\infty\\
&=\frac{(-q^2;q^2)_\infty}{(-q^{2t};q^{2t})_\infty}.
\end{align*}
A similar computation yields the even case.
\end{proof}

\begin{proof}[Proof of Corollary~\ref{coro:NOBGzt}]
Choose $\rho_1(h)=(1-u/h^2)^{1/2}$ in Theorem~\ref{thm:addition-multz}, then by~\eqref{eq:no} we get 
$$f_t(q)=(q;q)_\infty^{u/t^2-1}.$$
Next taking $\rho_2(h)=1$ and using~\eqref{eq:rho2=1}, Formula~\eqref{eq:addz} becomes
\begin{multline*}
\sum_{\lambda\in\BGP_{z,t}}q^{|\lambda|}x^{|\mathcal{H}_t(\lambda)|}|\mathcal{H}_t(\lambda)|\prod_{h\in\mathcal{H}_{t}(\lambda)}\left(1-\frac{u}{h^2}\right)^{1/2}=x\lfloor(t-z)/2\rfloor(f_t(x^2q^{2t}))^{\lfloor\frac{t-z}{2}\rfloor-1}\frac{d}{dx}f_t(x^2q^{2t})\\
\times\prod_{i=0}^{\lfloor (t-z)/2\rfloor-1}(-q^{2i+z+1},-q^{2t-2i-z-1},q^{2t};q^{2t})_\infty.
\end{multline*}
The desired result follows by dividing both sides by $x$,  integrating with respect to $x$, and setting $x=1$.
\end{proof}

\section*{Acknowledgements}
The first author is supported by the Agence Nationale de la Recherche funding COMBIN\'E  ANR-19-CE48-0011. 
The second author is supported by the Agence Nationale de la Recherche funding CORTIPOM ANR-21-CE40-001.


\end{document}